\documentclass[twoside,leqno,12pt]{article}

\usepackage{amssymb}
\usepackage{amsmath}
\usepackage{amsthm}
\usepackage{amscd}

\usepackage{color}

\definecolor{darkgreen}{rgb}{0,0.5,0}

\numberwithin{equation}{section}
\newtheorem{thm}[equation]{\sc Theorem}
\newtheorem{lem}[equation]{\sc Lemma}
\newtheorem{cor}[equation]{\sc Corollary}
\newtheorem{prop}[equation]{\sc Proposition}

\newtheoremstyle{notation}{3pt}{3pt}{}{}{\itshape}{:}{.5em}{\thmname{#1}}
\theoremstyle{notation}

\newtheorem{rem}{\it Remark}

\newtheorem{defin}{\it Definition}
\newtheorem{ex}{\it Example}

\def\mod{\mbox{{\rm mod}}}

\def\End{\mbox{\rm End}}

\def\rad{\mbox{\rm rad}}
\def\soc{\mbox{\rm soc\,}}

\input prepictex   \input pictex    
\input postpictex

\newcounter{boxsize}
\newcounter{tempcounter}
\newcommand{\smallentryformat}{\scriptstyle\sf}

\newcommand\smbox{\put(0,0){\line(1,0){\value{boxsize}}}%
  \put(\value{boxsize},0){\line(0,1){\value{boxsize}}}%
  \put(0,0){\line(0,1){\value{boxsize}}}%
  \put(0,\value{boxsize}){\line(1,0){\value{boxsize}}}}
\newcommand\numbox[1]{\put(0,0)\smbox%
  \put(0,0){\makebox(\value{boxsize},\value{boxsize})[c]{%
      $\smallentryformat#1$}}}
\newcommand\singlebox[1]{\raisebox{-.4ex}{\begin{picture}(4,0)\setcounter{boxsize}{3}%
    \put(0,0)\smbox%
    \put(0,0){\makebox(\value{boxsize},\value{boxsize})[c]{%
      $\scriptstyle\sf#1$}}\end{picture}}}
\newcommand\boxwithrow[2]{\raisebox{-.4ex}{\begin{picture}(5,0)\setcounter{boxsize}{3}%
    \put(0,0)\smbox%
    \put(0,0){\makebox(\value{boxsize},\value{boxsize})[c]{%
      $\scriptstyle\sf#1$}}
    \put(\value{boxsize},1.5){\line(1,0){1}}
    \put(\value{boxsize},1){$\;\scriptscriptstyle\sf#2$}
    \end{picture}}}

\def\sbullet{\makebox(0,0){$\scriptstyle\bullet$}}
\newcommand\boxes[2]{\ifthenelse{#2=3}{$\scriptstyle P_2^{#1}$}{%
                                       $\scriptstyle P_{#2}^{#1}$}}

\begin{document}
\thispagestyle{empty}
\color{black}
\phantom m\vspace{-2cm}

\bigskip\bigskip
\begin{center}
{\large\bf Finite direct sums of cyclic embeddings} 
\end{center}

\smallskip

\begin{center}
Justyna Kosakowska and Markus Schmidmeier

\bigskip
    {\it Dedicated to Professor Jos\'e Antonio de la Pe\~na}

\vspace{1cm}

\bigskip \parbox{10cm}{\footnotesize{\bf Abstract:}
  In this paper we generalize Kaplansky's combinatorial
  characterization of the isomorphism types of embeddings
  of a~cyclic subgroup in a finite abelian group given in his 1951 book
  {\it ``Infinite Abelian Groups''.} 
  For this we introduce partial maps on Littlewood-Richardson tableaux
  and show that they characterize the isomorphism types 
  of finite direct sums of such cyclic embeddings.}

\medskip \parbox{10cm}{\footnotesize{\bf MSC 2010:}
  05E10, 
  20K27, 
  47A15  

}

\medskip \parbox{10cm}{\footnotesize{\bf Key words:} 
subgroup embedding; invariant subspace; Littlewood-Richardson tableau; partial map }
\end{center}

\section{Introduction}

In \cite{birkhoff} Birkhoff stated the problem of classification of subgroups of finite abelian groups.
Connected with this problem are results of Kaplansky in \cite{kap} 
that present a~combinatorial
  characterization of the isomorphism types of embeddings of 
  a~cyclic subgroup in a~finite abelian group. This characterization uses sequences of positive integers, see Section \ref{sec-height-seq} for details.
  In this paper we use partial maps on Littlewood-Richardson tableaux (LR-tableaux for short)
  to generalize this result to finite direct sums of such embeddings. We extend results given in \cite[Section 2]{ks-box} and present some applications.\medskip

  Instead of finite abelian groups we work with finitely generated modules over 
  a~discrete valuation ring $\Lambda$. One of the presented result is Theorem \ref{theorem-combinatorial-classification} saying that a~direct sum of cyclic embeddings is uniquely determined by the associated LR-tableau with a~partial map satisfying the empty box property, see Section \ref{sec-tableaux} for definitions.
  LR-tableaux were successfully used in \cite{ks-hall,ks-survey,ks-box,ks-poles} to describe algebraic and geometric properties of short exact sequences of $\Lambda$-modules.
  Results obtained there show that relationships between short extact sequences of $\Lambda$-modules and LR-tableaux are deep 
  and interesting. 
  This motivates us to collect properties of LR-tableaux of direct sums of cyclic embeddings. \medskip

  The paper is organized as follows.

  \medskip
  In Section \ref{sec-poles} we present notation,
  definitions and basic facts concerning cyclic embeddings.
  In particular, we recall results of Kaplansky and
  rewrite them in the language of Littlewood-Richardson tableaux.

  \smallskip
  In Section \ref{sec-sums-cyclic} we formulate and prove
  the main result which is Theorem~\ref{theorem-combinatorial-classification}.

  \smallskip
  In Section \ref{sec-applications} we present applications
  of Theorem \ref{theorem-combinatorial-classification}.
  In particular, Corollary~\ref{cor-horizontal-strip}
  states that each LR-tableau of a certain shape
  can be realized as the LR-tableau of a direct sum of cyclic embeddings.
  This property will allow us in \cite{ks-box} to study
  the partition of the variety
  of short exact sequences of nilpotent linear operators
  into its irreducible components.

\section{Direct sums of cyclic embeddings}
\label{sec-poles}

 Let $\Lambda$ be a discrete valuation ring with maximal ideal generator $p$
and radical factor field $k$.
By $\mod\Lambda$ we denote the category of all (finite length)
$\Lambda$-modules. It is well known
that any
 finite length $\Lambda$-module $X$ is isomorphic to a~direct sum 
 $$X\cong \Lambda/(p^{\alpha_1})\oplus \Lambda/(p^{\alpha_2})\oplus \ldots\oplus \Lambda/(p^{\alpha_n}),$$
 where $\alpha_1\geq \alpha_2\geq\ldots\geq \alpha_n$. 
 The assignment 
 $X\mapsto \alpha=(\alpha_1,\ldots,\alpha_n)$ defines a bijection
 between the set of isomorphism classes of finite length $\Lambda$-modules 
 and the set of all partitions. 
 Given a partition $\alpha$, we denote by $N_\alpha(\Lambda)=N_\alpha$
 the corresponding $\Lambda$-module.\medskip
 
Denote by $\mathcal S(\Lambda)$ the category of all short exact sequences
in $\mod\Lambda$ with morphisms given by commutative diagrams.  
With the componentwise exact structure,
$\mathcal S(\Lambda)$ is an exact Krull-Remak-Schmidt category.
We denote objects in $\mathcal S(\Lambda)$ either as short exact sequences
$0\to A\to B\to C\to 0$ of $\Lambda$-modules, or as embeddings
$(A\subset B)$.

\subsection{Cyclic embeddings and poles}
\label{sec-height-seq}

\medskip
We review Kaplansky's classification of cyclic embeddings.

\medskip
An embedding $(A\subset B)$ is called {\it cyclic}
if the submodule $A$ is cyclic as a $\Lambda$-module,
that is, if $A_\Lambda$ is either indecomposable (i.e. of the form $\Lambda/(p^n)$ for some $n\geq 1$) or zero.
A cyclic embedding $(A\subset B)$ 
is a {\it pole} if $A$ is an indecomposable $\Lambda$-module
and if $(A\subset B)$ is an indecomposable object in $\mathcal S(\Lambda)$.
An embedding of the form $(0\subset B)$ is called {\it empty.} 
By $E_\beta$ we denote the empty embedding
  $(0\subset N_\beta)$.\medskip

For the classification we use height sequences which were   
originally introduced by Pr\"ufer
\cite{prue} as {\it H\"ohenexponenten.}
In \cite[Section~18]{kap}, 
the height sequence for an~element $a\in B$ is called 
its {\it Ulm-sequence}.

\begin{defin}\begin{itemize}
  \item A {\it height sequence} is a~sequence in $\mathbb N_0\cup\{\infty\}$
      which is strictly increasing until it reaches $\infty$
      after finitely many steps.
    A~height sequence is {\it non-empty}, if it has at least
    one element from $\mathbb{N}_0$.
  \item Let $B\in\mod\Lambda$. We say an element $a\in B$ has {\it height} $m$
    if $a\in p^m B\setminus p^{m+1}B$.
    In this case we write $h(a)=m$.
    By definition, $h(0)=\infty$.
  \item   The {\it height sequence for $a$ in $B$} is 
    $H_B(a)=(h(p^ia))_{i\in\mathbb N_0}$.
    Usually, we do not write the entries $\infty$.
  \item  A~sequence 
$(m_i)$ has a~{\it gap} after $m_\ell$ if $m_{\ell+1}>m_\ell+1$.  
  \end{itemize}
\end{defin}

The following result is proved in \cite[Theorem 25]{kap}.

\begin{prop}\label{prop-poles}
There is a~bijection
$$\{\text{poles}\}/_{\cong} \quad \stackrel{1-1}\longleftrightarrow \quad 
\{\text{finite non-empty strictly increasing sequences in $\mathbb N_0$}\}.$$
\end{prop}

\begin{ex}
Following \cite{kap}, 
we construct a~cyclic embedding $(A\subset B)$ 
for a~given strictly increasing
sequence $(m_i)_{0\leq i\leq n}$ in $\mathbb N_0$. 
Let $i_1>i_2>\cdots>i_s$ be  such that 
$(m_i)$  has gaps exactly after the entries 
$m_{i_1}>m_{i_2}>\cdots>m_{i_s}$.
For $1\leq j\leq s$ put $\beta_j=m_{i_j}+1$
and $\ell_j=m_{i_j}-i_j$, then $\beta$ and $\ell$ are strictly decreasing
sequences of positive and nonnegative integers, respectively.  
Let $$B=N_\beta=\bigoplus_{i=1}^s \Lambda/(p^{\beta_i})$$
be generated by elements $b_{\beta_j}$ of order $p^{\beta_j}$.
Let $a=\sum_{j=1}^s p^{\ell_j}\cdot b_{\beta_j}$ and put $A=(a)$. 
This yields a cyclic embedding $P((m_i))=(A\subset B)$.
 
\end{ex}

\begin{ex}
  The height sequence $(1,3,4)$ has gaps after $1$ and $4$, and
  hence gives rise to the embedding
  $P((1,3,4))=((p^2b_5+pb_2)\subset N_{(5,2)})$.
  $$
  \raisebox{.5cm}{$P((1,3,4)):$} \quad
\begin{picture}(8,15)
  \multiput(0,0)(0,3)5{\smbox}
  \multiput(3,6)(0,3)2{\smbox}
  \multiput(1.5,7.5)(3,0)2\sbullet
  \put(1.5,7.5){\line(1,0)3}
\end{picture}
$$
In the picture, the columns correspond to the indecomposable direct 
summands of $B$; the $i$-th box from the top in a column of length $r$
represents the element $p^{i-1}b_r$.  The columns are aligned vertically
to make the submodule generator(s) homogeneous, if possible.
\end{ex}

\begin{lem}[{see \cite[Lemma~22]{kap}}]
  \label{lemma-gap}
  Suppose the height sequence $(m_i)$ of some element $a$ in $B$ has a gap after $m_\ell$.
  Then $N_{(m_\ell+1)}$ occurs as a direct summand of $B$.
\end{lem}

\begin{proof}[Proof of the lemma]
 See \cite[Lemma~22 and page 29]{kap}.
\end{proof}

The following result is a slight generalization of Proposition \ref{prop-poles}. We will use  it later.

\begin{prop}\label{prop-cyclics}
  Every cyclic embedding in $\mathcal S(\Lambda)$ has the form
  $$P((m_i))\oplus E_\beta$$
  for a~finite, possibly empty, strictly increasing 
  sequence $(m_i)$ and a possibly empty partition $\beta$.
  More precisely, there is a one-to-one correspondence
  between the set of cyclic
  embeddings up to isomorphy and the set of pairs consisting of a~finite,
  possibly empty, strictly increasing
  sequence and a partition.
\end{prop}

\begin{proof}
 Let $(A\subset B)$ be a cyclic embedding in $\mathcal S(\Lambda)$.  
    The category $\mathcal S(\Lambda)$ has the Krull-Remak-Schmidt property
    so $(A\subset B)$ is isomorphic to a direct sum of indecomposable
    cyclic embeddings.
    As $\Lambda$ is a local ring, and as the submodule
    $A$ is 
    indecomposable or zero, at most
    one of the summands has a nonzero submodule, hence is a pole,
    and hence is given uniquely, up to isomorphy,
      by a finite non-empty strictly increasing sequence in $\mathbb N_0$
      (Proposition~\ref{prop-poles}).
    (The case where the submodule $A$ is zero is included
      since the embedding
       given by the empty height sequence
    denotes the zero object in $\mathcal S(\Lambda)$.)
    The sum of the remaining indecomposable summands has the form $(0\subset B')$
    where  $B'$ is a $\Lambda$-module such that $(0\subset B')\cong E_\beta$
      for some partition $\beta$.
\end{proof}

  \begin{rem}
    Each indecomposable cyclic embedding is either a pole of the form
    $P((m_i))$ for $(m_i)$ a finite non-empty strictly increasing
    sequence in $\mathbb N_0$ or else an empty embedding
    of the form $E_{(n)}$ for $n$ a natural number.
  \end{rem}

\bigskip

\subsection{Partitions, tableaux and partial maps} 

\label{sec-tableaux}
We introduce partial maps on Littlewood-Richardson tableaux and
review cyclic embeddings.

\begin{defin}
Let $(\alpha,\beta,\gamma)$ be a~partition triple. An {\it LR-tableau} of shape $\beta\setminus\gamma$ and content $\alpha$ is a~Young diagram of shape $\beta$
(the {\it outer shape} of the tableau) in which 
the region $\beta\setminus\gamma$ is filled with $(\alpha')_1$ entries
1, $(\alpha')_2$ entries 2, etc., where $\alpha'$ denotes the transpose 
of $\alpha$, such that three conditions are satisfied:
\begin{itemize}
 \item in each row, the entries are weakly increasing;
 \item in each column, the entries
are strictly increasing;
\item the lattice permutation property holds, that is,
on the right hand side of each column there occur at least as many entries
$e$ as there are entries $e+1$ ($e=1,2,\ldots$).
\end{itemize}
\end{defin}

\medskip
An embedding $(A\subset B)$ with corresponding short exact
sequence $0\to A\to B\to C\to 0$ yields three partitions
$\alpha,\beta,\gamma$ which describe the isomorphism types of 
$A,B,C$, respectively.
The Green-Klein Theorem \cite{klein1}  states that a partition triple
$(\alpha,\beta,\gamma)$ can be realized in this way by a short exact 
sequence if and only if there exists an LR-tableau
of shape $\beta\setminus\gamma$ and content $\alpha$.
The tableau of an embedding can be obtained in the following way.

\medskip

Suppose that the submodule $A$ has Loewy
length $r$.  Then the sequence of epimorphisms
$$B=B/p^rA\twoheadrightarrow B/p^{r-1}A\twoheadrightarrow \cdots\twoheadrightarrow B/pA
\twoheadrightarrow B/A$$
induces a~sequence of inclusions of partitions
$$\beta=\gamma^{(r)}\supset\gamma^{(r-1)}\supset\cdots
  \supset\gamma^{(1)}\supset\gamma^{(0)}=\gamma,$$
where  $N_{\gamma^{(e)}}\simeq B/p^eA$.
The partitions define the corresponding LR-tableau
$\Gamma=[\gamma^{(0)}, \gamma^{(1)},\ldots,\gamma^{(r)}]$ as follows.
The Young diagram of shape $\beta=\gamma^{(r)}$ is filled in such a way
that the region given by $\gamma=\gamma^{(0)}$ remains empty and 
every box in the skew diagram
$\gamma^{(e)}\setminus\gamma^{(e-1)}$ carries the entry $\singlebox e$.
The {\it shape} of the tableau is the skew diagram $\beta\setminus\gamma$,
the {\it content} is the partition $\alpha$.

\medskip
The {\it union} of two tableaux is taken row-wise, so if
${\rm E}=[\varepsilon^{(i)}]_{0\leq i\leq s}$ and ${\rm Z}=[\zeta^{(i)}]_{0\leq i\leq t}$ 
are tableaux then
the partitions $\gamma^{(i)}$ for $\Gamma={\rm E}\cup{\rm Z}$ are given by taking
the union (of ordered multi-sets):  
$\gamma^{(i)}=\varepsilon^{(i)}\cup\zeta^{(i)}$ 
where $\varepsilon^{(i)}=\varepsilon^{(s)}$
for $i\geq s$ and $\zeta^{(i)}=\zeta^{(t)}$ for $i\geq t$.

\begin{rem}
It is easy to observe that if embeddings $(A\subset B)$, $(C\subset D)$
have tableaux $\Gamma_1$, $\Gamma_2$, respectively, then the direct sum
$(A\subset B)\oplus (C\subset D)$ has tableau 
$\Gamma=\Gamma_1\cup \Gamma_2$. 
\end{rem}


\begin{defin}
{\it Column tableaux} (which we call {\it columns} if the context is clear) 
are tableaux which have only one column, and where the entries
in this column are subsequent natural numbers (not necessarily starting at 1).
Each column is one of the following: 
For $1\leq e\leq f$, we denote by $C(e,f)_n$ 
the 1-column tableau of height $n$ with entries $e,\ldots,f$.
Formally, $C(e,f)_n=[\gamma^{(0)},\ldots,\gamma^{(f)}]$ where
$\gamma^{(0)}=\cdots=\gamma^{(e-1)}=(n-f+e-1), \gamma^{(e)}=(n-f+e),\ldots,\gamma^{(f)}=(n)$.  
By $C(1,0)_n$ we denote the empty column of length $n$.
Note that a column tableau $C(e,f)_n$ satisfies
the lattice permutation property, and hence is an LR-tableau,
  if and only if $e=1$. 
\end{defin}

\begin{ex}
In a union of column tableaux, the vertical columns
need no longer form  column tableaux.
  $$
  \begin{picture}(3,9)
    \multiput(0,0)(0,3)3{\smbox}
    \put(0,6){\numbox1}
    \put(0,3){\numbox2}
    \put(0,0){\numbox3}
  \end{picture} \quad
  \raisebox{.35cm}{$\cup$} \quad
  \begin{picture}(3,9)
    \multiput(0,3)(0,3)2{\smbox}
    \put(0,3){\numbox1}
  \end{picture}
  \quad\raisebox{.35cm}{$=$} \quad
  \begin{picture}(6,9)
    \multiput(0,0)(0,3)3{\smbox}
    \multiput(3,3)(0,3)2{\smbox}
    \put(3,6){\numbox1}
    \put(3,3){\numbox2}
    \put(0,3){\numbox1}
    \put(0,0){\numbox3}
  \end{picture}
  $$
\end{ex}

\begin{defin}
  A {\it partial map} $g$ on an LR-tableau $\Gamma$ assigns to each box
  $\singlebox e$ with $e>1$ a box $\singlebox d$ with entry $d=e-1$ such that
  \begin{enumerate}
  \item $g$ is one-to-one,
  \item for each box $b$, the row of $g(b)$ is above the row of $b$. 
  \end{enumerate}
    An {\it orbit} of a partial map $g$ is a sequence 
    $\singlebox e, g(\singlebox e), g^2(\singlebox e),\ldots
    , g^{e-1}(\singlebox e)$ given by a box $\singlebox e$ 
      which is not in the image of $g$.
      Thus, the orbits of $g$ are in one-to-one correspondence with
      the boxes $\singlebox e$ not in the image of $g$ (the correspondence
      is given by taking the first box in the orbit),
      and in one-to-one correspondence with the boxes that have entry 1
      (where the correspondence is given by taking the last box).
\end{defin}

\begin{rem}
  Given an LR-tableau $\Gamma$, the existence of at least one partial map
  follows from the lattice permutation property.
\end{rem}

Given a partial map $g$ on an LR-tableau $\Gamma$, a {\it jump} in row $r$ 
is a box $b$ in this row with the property that either $b$ 
has entry $\singlebox 1$ or that $g(b)$ occurs in row $s$ with $s<r-1$. 
We say a partial map $g$ on $\Gamma$ has the {\it empty box property (EBP)}
if for each row $r$, there are at least as many columns in 
$\Gamma$ of exactly $r-1$ empty boxes as there are jumps in row $r$.

\begin{ex}
Consider the following LR-tableau
$$
 \raisebox{.5cm}{$\Gamma:$} \quad
  \begin{picture}(12,12)
    \multiput(0,0)(0,3)4{\smbox}
    \multiput(3,3)(0,3)3{\smbox}
    \multiput(6,6)(0,3)2{\smbox}
    \put(9,9){\smbox}
    \put(9,9){\numbox1}
    \put(6,9){\numbox1}
    \put(6,6){\numbox2}
    \put(0,3){\numbox1}
    \put(3,3){\numbox2}
    \put(0,0){\numbox3}
  \end{picture}
$$  
To define a~partial map, we have to assign to each box 
$\singlebox 3$ a corresponding box $\singlebox 2$ and  to each box $\singlebox 2$ 
a corresponding box $\singlebox 1$. 
Note that we can do this in four different ways. 
To specify the maps, we distinguish boxes 
of the same entry by indicating the row on the right;
for the two boxes in the first row, we use the letters
 L and R.
Consider the partial map $g$ defined as follows:
$$g:\boxwithrow34\mapsto\boxwithrow23,\; \boxwithrow23\mapsto \boxwithrow1L,\;
\boxwithrow22\mapsto \boxwithrow1R$$
Note that the partial map $g$ has three orbits
(one consists only of the box $\boxwithrow 13$).
It satisfies (EBP) since there is a column of 2 empty boxes
corresponding to the jump $\boxwithrow23\mapsto \boxwithrow 11$.
  However, the partial map
$$g': \boxwithrow34\mapsto \boxwithrow22,\; \boxwithrow23\mapsto \boxwithrow1L,\;
\boxwithrow22\mapsto \boxwithrow1R$$
does not satisfy (EBP), because there is no column of 3 empty boxes
corresponding to the jump at $\boxwithrow 34\mapsto \boxwithrow22$.

\end{ex}

  \medskip
We collect some properties of tableaux for cyclic embeddings.

\begin{prop}
  \label{prop-tableau-of-cyclic}
  Suppose the embedding $(A\subset B)$ is cyclic with $A$ a module
  of Loewy length $r$.
  \begin{enumerate}
  \item In the tableau, each entry $1, \ldots, r$
    occurs exactly once.
  \item The height sequence $(m_i)$
    of the submodule generator $a$ 
    determines the rows in $\Gamma$ in which
    the entries occur, and conversely.  
    More precisely, the entry $e$ occurs in row $m_{e-1}+1$.
  \item The LR-tableau $\Gamma$ of a cyclic embedding 
    is a union of columns with disjoint entries.
  \item There is a unique partial map on $\Gamma$; it satisfies the (EBP).
  \end{enumerate}
\end{prop}

\begin{proof}
  1. For the first statement note that for each $0<e\leq r$,
  the module $B/p^{e-1}A$ is a factor module of $B/p^eA$ of colength one.

  2. For any embedding, the number of boxes in the first $m$ rows of $\gamma^{(e)}$
  is the length of $(B/p^eA)/((p^mB+p^eA)/p^eA)=B/(p^eA+p^mB)$.
  Here, the modules $B/(p^{e-1}A+p^mB)$ and $B/(p^eA+p^mB)$ have the same length
  if $m\leq m_{e-1}$ and different lengths otherwise since $p^{e-1}a\in p^{m_{e-1}}B\setminus p^{m_{e-1}+1}B$.

  3. By construction of $P((m_i))$, each gap in the height sequence gives rise to a
  new summand of $B$ and hence to a new column in the tableau.
  It follows from 2.\ that $\Gamma$ is a union of columns.
  
  4. There is a unique partial map on $\Gamma$ given by assigning to a box $\singlebox e$
  where $e>1$ the unique box with entry $e-1$.  Since $\Gamma$ is a union of columns,
  the (EBP) is satisfied.
\end{proof}

\begin{cor}\label{cor-tableaux}%
\begin{enumerate}
\item There is a one-to-one correspondence between cyclic embeddings
  up to isomorphy and LR-tableaux in which each entry occurs exactly once.
\item Under this correspondence, a pole corresponds to a tableau 
  in which the number of columns equals the number of gaps in the 
  height sequence.
  An empty embedding $E_\beta$
    corresponds to the empty tableaux on the Young diagram
    for $\beta$.
\end{enumerate}
\end{cor}

\begin{proof}
  1. By the Green-Klein theorem, any LR-tableau $\Gamma$ can be realized 
  as the tableau of an embedding $(A\subset B)$.  
  If $\Gamma$ has exactly one box $\singlebox 1$, then $A$ is 
an indecomposable $\Lambda$-module
  since $\dim A/pA=\dim B/pA-\dim B/A=\#\{\text{boxes $\singlebox 1$}\}$.
  The height sequence of the generator of $A$ can be read off
  using Proposition~\ref{prop-tableau-of-cyclic}, 2. The uniqueness of the
  isomorphism type follows from Proposition~\ref{prop-cyclics}.

  2. Regarding poles, we have seen in the proof of Proposition~\ref{prop-poles}
  that the summands of the ambient space of a pole
  correspond to the gaps in the height sequence
  of the submodule generator.  The result follows from 
  Proposition~\ref{prop-cyclics}.
\end{proof}

 \section{Direct sums of cyclic embeddings}

\label{sec-sums-cyclic}

We can now generalize Kaplansky's results and
give a~combinatorial description for 
finite direct sums of cyclic embeddings, up to isomorphy.

\begin{defin}\begin{itemize}  
  \item
    Two partial maps $g$, $g'$ on an LR-tableau $\Gamma$ are {\it equivalent}
    if $g'=\pi^{-1}\circ g\circ \pi$ for some permutation $\pi$ 
    on the set of non-empty  boxes  in $\Gamma$
    which preserves the entry and the row.
  \item Two pairs $(\Gamma,g)$, $(\Delta, h)$, each consisting of an
    LR-tableau and a partial map on the tableau, are {\it equivalent}
    if $\Gamma=\Delta$ and if the partial maps $g,h$ are equivalent.
    In this case we write $(\Gamma,g)\sim(\Delta,h)$.
  \end{itemize}
\end{defin}

 \begin{ex}
  Consider the LR-tableau from the example in Section \ref{sec-tableaux}.
$$
 \raisebox{.5cm}{$\Gamma:$} \quad
  \begin{picture}(12,12)
    \multiput(0,0)(0,3)4{\smbox}
    \multiput(3,3)(0,3)3{\smbox}
    \multiput(6,6)(0,3)2{\smbox}
    \put(9,9){\smbox}
    \put(9,9){\numbox1}
    \put(6,9){\numbox1}
    \put(6,6){\numbox2}
    \put(0,3){\numbox1}
    \put(3,3){\numbox2}
    \put(0,0){\numbox3}
  \end{picture}
$$  
The maps $g$ and $g'$ are not equivalent.  
But $g$ is equivalent to the partial map
$$h: \singlebox3\mapsto\boxwithrow23, \;\boxwithrow23\mapsto \boxwithrow1R,\;
\boxwithrow22\mapsto \boxwithrow1L.$$
(There is a fourth partial map on $\Gamma$, it is equivalent to $g'$.)
 \end{ex}

\medskip
Clearly, given two equivalent partial maps on an LR-tableau,
then one satisfies (EBP) if and only if the other does.

\begin{thm}\label{theorem-combinatorial-classification}
  There is a one-to-one correspondence
    $$\big\{\bigoplus\text{cyclics}\big\} \big /_{\cong}
    \quad \stackrel{1-1}\longleftrightarrow \quad
    \big\{(\Gamma,g):\text{$g$ partial map on $\Gamma$ with (EBP)}\big\}
    \big /_{\sim}$$
  between the isomorphism types
  of direct sums of cyclic embeddings, and the equivalence classes
  of pairs $(\Gamma,g)$ where $\Gamma$ is an LR-tableau
  and $g$ a partial map on $\Gamma$ which satisfies (EBP).
\end{thm}

\begin{proof}
    We describe this correspondence explicitely.

    \smallskip
    Consider a direct sum of cyclic embeddings as the sum
    of an empty embedding and a direct sum of poles
    (Proposition~\ref{prop-cyclics}).
    The height sequence of each pole gives rise to a partial map
    on the tableau of the pole which satisfies (EBP),
    this map has precisely one orbit; it records the isomorphism
    type of the pole (Proposition~\ref{prop-tableau-of-cyclic}).
The tableau of the direct sum is given by taking the row-wise union
  of the tableaux of the summands; it admits a partial map which is 
  (boxwise) defined by the partial maps for the summands; this map satisfies
  (EBP) because the restricted maps do. 
  Given two isomorphic direct sum decompositions, the two unordered lists of
  height sequences of the poles involved are equal, hence the 
  two associated partial maps differ by conjugation by a permutation of the boxes
  which preserves rows and entries.

\smallskip
Conversely, suppose $g$ is a partial map on $\Gamma$ with (EBP).
For each orbit $\mathcal O$ of $g$, the boxes in $\mathcal O$ together
with the empty parts of columns which correspond to the jumps 
(recall the definition of the (EBP)), 
constitute the tableau of a pole.  
The empty parts of columns in $\Gamma$
not used by the (EBP) define the tableau of an empty embedding. 
Since all boxes are accounted for, $\Gamma$ is the union of all the 
tableaux.  Hence the sum of the poles and the empty embedding 
has $\Gamma$ as its tableau.
If $g'$ is an equivalent partial map on $\Gamma$, then the boxes in the 
orbits may differ, but not the rows in which they occur.
Thus the corresponding
unordered list of height sequences of poles is equal, 
and so is the partition
which records those empty parts of columns which are not used up
by the jumps in the height sequences.  
Hence the associated embeddings are isomorphic.

\smallskip
The two assignments are inverse to each other since this is true for 
poles and empty embeddings.
\end{proof}

There is a different way to read off from a given LR-tableau
if it is the tableau of a direct sum of cyclic embeddings.

\begin{cor}
\label{cor-union-columns}
An LR-tableau $\Gamma$ is the tableau of 
a direct sum of cyclic embeddings if and only if $\Gamma$ is
a union of columns.
\end{cor}

\begin{proof}
In Proposition~\ref{prop-tableau-of-cyclic} 
we have seen that the tableau of a cyclic embedding  is 
a union of columns, and this property is preserved under taking
direct sums.

\smallskip
Conversely, if the tableau is a union of columns, then the 
lattice permutation property allows for the construction of a 
partial map with (EBP).  The result follows from 
Theorem~\ref{theorem-combinatorial-classification}.
\end{proof}

We conclude this section with examples and remarks.

\begin{ex}
In general, the embedding corresponding to a given LR-tableau
is not determined uniquely, up to isomorphy, by the tableau alone.
In each case, the following embeddings will have the same given LR-tableau.

\medskip
(1) Two nonisomorphic sums of poles, given by nonequivalent partial maps.
  $$
  \raisebox{.5cm}{$\Gamma:$} \quad
  \begin{picture}(8,12)
    \multiput(0,0)(0,3)4{\smbox}
    \multiput(3,6)(0,3)2{\smbox}
    \put(6,9){\numbox1}
    \put(3,6){\numbox1}
    \put(0,0){\numbox2}
  \end{picture}
  \qquad
  \raisebox{.5cm}{$P((1,3))\oplus P((0)):$} \quad
  \begin{picture}(12,12)
    \multiput(0,0)(0,3)4{\smbox}
    \multiput(3,3)(0,3)2{\smbox}
    \multiput(1.5,4.5)(3,0)2\sbullet
    \put(1.5,4.5){\line(1,0)3}
    \put(9,3){\smbox}
    \put(10.5,4.5)\sbullet
  \end{picture}
  \qquad
  \raisebox{.5cm}{$P((0,3))\oplus P((1)):$} \quad
  \begin{picture}(12,12)
    \multiput(0,0)(0,3)4{\smbox}
    \multiput(3,3)(0,3)1{\smbox}
    \multiput(1.5,4.5)(3,0)2\sbullet
    \put(1.5,4.5){\line(1,0)3}
    \multiput(9,3)(0,3)2{\smbox}
    \put(10.5,4.5)\sbullet
  \end{picture}
  $$

(2) A sum of two poles vs.\ a sum of two poles and an empty embedding.
  $$
  \raisebox{.5cm}{$\Gamma:$} \quad
  \begin{picture}(8,9)
    \multiput(0,0)(0,3)3{\smbox}
    \multiput(3,6)(0,3)1{\smbox}
    \put(6,6){\numbox1}
    \put(3,3){\numbox1}
    \put(0,0){\numbox2}
  \end{picture}
  \qquad
  \raisebox{.5cm}{$P((0,2))\oplus P((1)):$} \quad
  \begin{picture}(12,9)
    \multiput(0,0)(0,3)3{\smbox}
    \multiput(3,3)(0,3)1{\smbox}
    \multiput(1.5,4.5)(3,0)2\sbullet
    \put(1.5,4.5){\line(1,0)3}
    \multiput(9,3)(0,3)2{\smbox}
    \put(10.5,4.5)\sbullet
  \end{picture}
  \qquad
  \raisebox{.35cm}{\parbox{.23\textwidth}{%
      $P((1,2))\oplus P((0))$\\ $\phantom x\quad\oplus E_{(2)}:$}} 
  \quad
  \begin{picture}(15,9)
    \multiput(0,0)(0,3)3{\smbox}
    \multiput(6,3)(0,3)1{\smbox}
    \multiput(1.5,4.5)(3,0)1\sbullet
     \multiput(12,3)(0,3)2{\smbox}
    \put(7.5,4.5)\sbullet
  \end{picture}
  $$

(3) The sum of poles is determined uniquely by the partial map; 
but there is also an embedding which is not a sum of poles.
  $$
  \raisebox{.5cm}{$\Gamma:$} \quad
  \begin{picture}(8,12)
    \multiput(0,0)(0,3)4{\smbox}
    \multiput(3,3)(0,3)3{\smbox}
    \multiput(6,6)(0,3)2{\smbox}
    \put(6,9){\numbox1}
    \put(6,6){\numbox2}
    \put(3,3){\numbox1}
    \put(0,0){\numbox3}
  \end{picture}
  \qquad
  \raisebox{.5cm}{$P((0,1,3))\oplus P((2)):$} \quad
  \begin{picture}(12,12)
    \multiput(0,0)(0,3)4{\smbox}
    \multiput(3,3)(0,3)2{\smbox}
    \multiput(1.5,7.5)(3,0)2\sbullet
    \put(1.5,7.5){\line(1,0)3}
    \multiput(9,3)(0,3)3{\smbox}
    \put(10.5,4.5)\sbullet
  \end{picture}
  \qquad
  \raisebox{.5cm}{$E\oplus E_{(3)}:$} \quad
  \begin{picture}(12,12)
    \multiput(0,0)(0,3)4{\smbox}
    \multiput(3,3)(0,3)2{\smbox}
    \multiput(1.5,7.5)(3,0)2\sbullet
    \put(4.5,4.5)\sbullet
    \put(1.5,7.5){\line(1,0)3}
    \multiput(9,3)(0,3)3{\smbox}
  \end{picture}
  $$

(4) Two partial maps, up to equivalence, only one has the (EBP).
  $$
  \raisebox{.5cm}{$\Gamma:$} \quad
  \begin{picture}(12,12)
    \multiput(0,0)(0,3)4{\smbox}
    \multiput(3,3)(0,3)3{\smbox}
    \multiput(6,6)(0,3)2{\smbox}
    \put(9,9){\smbox}
    \put(9,9){\numbox1}
    \put(6,9){\numbox1}
    \put(6,6){\numbox2}
    \put(0,3){\numbox1}
    \put(3,3){\numbox2}
    \put(0,0){\numbox3}
  \end{picture}
  \qquad
  \raisebox{.5cm}{\parbox{.23\textwidth}{%
      $P((0,1,3))$\\ $\phantom x\quad \oplus P((2))$\\
      $\phantom x\quad \oplus P((0,1)):$}}
  \begin{picture}(18,12)
    \multiput(0,0)(0,3)4{\smbox}
    \multiput(3,3)(0,3)2{\smbox}
    \multiput(1.5,7.5)(3,0)2\sbullet
    \put(1.5,7.5){\line(1,0)3}
    \multiput(9,3)(0,3)3{\smbox}
    \put(10.5,4.5)\sbullet
    \multiput(15,3)(0,3)2{\smbox}
    \put(16.5,7.5)\sbullet
  \end{picture}
  \qquad
  \raisebox{.5cm}{$E\oplus P((0,2)):$} \quad
  \begin{picture}(15,12)
    \multiput(0,0)(0,3)4{\smbox}
    \multiput(3,3)(0,3)2{\smbox}
    \multiput(1.5,7.5)(3,0)2\sbullet
    \put(4.5,4.5)\sbullet
    \put(1.5,7.5){\line(1,0)3}
    \multiput(9,3)(0,3)3{\smbox}
    \put(12,6){\smbox}
    \multiput(10.5,7.5)(3,0)2\sbullet
    \put(10.5,7.5){\line(1,0)3}
  \end{picture}
  $$

(5) Two partial maps with (EBP), up to equivalence.
  $$
  \raisebox{.5cm}{$\Gamma:$} \quad
  \begin{picture}(12,12)
    \multiput(0,0)(0,3)4{\smbox}
    \multiput(3,3)(0,3)3{\smbox}
    \multiput(6,3)(0,3)3{\smbox}
    \put(9,9){\smbox}
    \put(9,9){\numbox1}
    \put(6,6){\numbox1}
    \put(3,3){\numbox2}
    \put(6,3){\numbox2}
    \put(0,0){\numbox3}
  \end{picture}
  \qquad
 \raisebox{.5cm}{\parbox{.23\textwidth}{%
      $P((0,2,3))$\\ $\phantom x\quad \oplus P((1,2))$\\
      $\phantom x\quad \oplus E_{(3)}:\;$}}
  \begin{picture}(18,12)
    \multiput(0,0)(0,3)4{\smbox}
    \multiput(3,6)(0,3)1{\smbox}
    \multiput(1.5,7.5)(3,0)2\sbullet
    \put(1.5,7.5){\line(1,0)3}
    \multiput(9,3)(0,3)3{\smbox}
    \put(10.5,7.5)\sbullet
    \multiput(15,3)(0,3)3{\smbox}
  \end{picture}
  \qquad
 \raisebox{.5cm}{\parbox{.23\textwidth}{%
      $P((1,2,3))$\\ $\phantom x\quad \oplus P((0,2))$\\
      $\phantom x\quad \oplus E_{(3)}:\;$}}
  \begin{picture}(12,12)
    \multiput(0,0)(0,3)4{\smbox}
    \put(1.5,7.5){\sbullet}
    \multiput(6,3)(0,3)3{\smbox}
    \put(9,6){\smbox}
    \multiput(7.5,7.5)(3,0)2\sbullet
    \put(7.5,7.5){\line(1,0)3}
    \multiput(15,3)(0,3)3{\smbox}
  \end{picture}
  $$

\end{ex}

\begin{rem}
The proof of Theorem~\ref{theorem-combinatorial-classification}
shows how to read off the direct sum decomposition from a given pair
$(\Gamma,g)$. 
Each orbit of $g$ gives rise to a pole, its height sequence 
is given by Proposition~\ref{prop-tableau-of-cyclic}.
To each jump in the height sequence, there is a corresponding empty part
of a column in $\Gamma$. 
The empty boxes in those columns which do not correspond to any pole
determine the remaining empty embedding.
\end{rem}

\begin{rem}
An algebraic description of finite direct sums of cyclic embeddings
is given in \cite[in particular Lemma~1 and Theorem~2]{hrw}.  
Here, the multiplicity of a pole $P$ as a direct
summand of a given embedding $M$ is the dimension of the $k$-vector
space $F_P(M)$, for a suitable functor $F_P:\mathcal S(\Lambda)\to \mod k$.
\end{rem}

\section{Applications of height sequences}
\label{sec-applications}

We briefly describe four situations in which height sequences
or direct sums of cyclic embeddings play a role.

\subsection{$p^2$-bounded embeddings in infinite length modules}

In this paragraph, $B$ may be any $\Lambda$-module, not necessarily
finitely generated.

\begin{cor}\label{cor-p-squared-bounded-sub}
Suppose that $(A\subset B)$ is an embedding of $\Lambda$-modules where 
$A$ is $p^2$-bounded and where $B$ is either finitely generated
or bounded.  Then the embedding is a direct sum of cyclic embeddings.
\end{cor}

\begin{proof}
The case where $B$ is finitely generated is covered in \cite[Corollary~5.4]{ps}.
There is exactly one indecomposable embedding where $B$ is not bounded,
namely $(0\subset\Lambda)$.

\smallskip
For the bounded case, note that an embedding $(A\subset B)$ with
$p^2A=0=p^nB$  can be considered a module over the ring
$$R = \begin{pmatrix} \Lambda/(p^2) & \Lambda/(p^2)\\ 0 & \Lambda/(p^n)
\end{pmatrix},$$
so \cite[Theorem~1]{moore} yields the claim.
\end{proof}

\subsection{Direct sums of cyclics in the Green-Klein Theorem}

Let $\alpha$, $\beta$, $\gamma$ be partitions.  
Recall from \cite{klein1} that there exists a short exact sequence 
$$\mathcal E: 0\to N_\alpha\to N_\beta\to N_\gamma\to 0$$
if and only if there exists an LR-tableau
of shape $\beta\setminus\gamma$ and content $\alpha$. 
As a consequence, we obtain

\begin{cor}
Given partitions $\alpha$, $\beta$, $\gamma$, there exists an embedding 
$N_\alpha\to N_\beta$ with cokernel $N_\gamma$ which is a direct sum of cyclic 
embeddings if and only if there is an LR-tableau of shape $\beta\setminus\gamma$
and content $\alpha$ which is a union of columns.
\end{cor}

\begin{proof}
Using Corollary~\ref{cor-union-columns} 
it is possible to decide whether there exists an embedding $N_\alpha\to N_\beta$ with
cokernel $N_\gamma$ which is a direct sum of cyclic embeddings.
\end{proof}

\subsection{Tableaux which are horizontal strips}

The skew diagram  $\beta\setminus\gamma$ 
is said to be a~{\it horizontal strip} if $\beta_i\leq \gamma_i+1$ 
holds for all $i$, and a~{\it vertical strip} if $\beta'\setminus\gamma'$
is a~horizontal strip.
We are particularly interested in the following situation.

\begin{cor}\label{cor-horizontal-strip}
Suppose $\Gamma$ is an LR-tableau of shape 
$(\alpha,\beta,\gamma)$ such that $\beta\setminus\gamma$ is a horizontal
strip.  
Then there exists an embedding $(A\subset B)$ which is 
a direct sum of cyclic embeddings and which has tableau $\Gamma$.
\end{cor}

\begin{proof}
Since $\Gamma$ is an LR-tableau, there exists a partial
map $g$ on $\Gamma$.  Since $\beta\setminus\gamma$ is a horizontal strip,
each column in the tableau $\Gamma$ has at most one entry, so the (EBP)
is trivially satisfied.
The embedding corresponding to the partial map $g$ is a direct sum
of cyclic embeddings, and it has the tableau $\Gamma$.
\end{proof}

\begin{rem}  In the situation of Corollary~\ref{cor-horizontal-strip},
finite direct sums of cyclic embeddings turn out to be a useful tool
for the study of the variety $\mathbb V_{\alpha,\gamma}^\beta$
of all embeddings $N_\alpha\to N_\beta$ with cokernel $N_\gamma$;
here, $\Lambda=k[[T]]$  is the power series ring with coefficients
in an algebraically closed field.
Namely, each irreducible component of $\mathbb V_{\alpha,\gamma}^\beta$
has the form $\overline{\mathbb V}_\Gamma$
for some LR-tableau  $\Gamma$  of shape $(\alpha,\beta,\gamma)$.
The boundary relation for tableaux $\Gamma$, $\Delta$ is given by
the condition that
$\overline{\mathbb V}_\Gamma\cap\mathbb V_\Delta\neq\emptyset$.
In  \cite[Theorem 5.1]{ks-box} it is shown that box moves in LR-tableaux
imply the boundary relation by constructing one-parameter
families of embeddings which are all finite direct sums of cyclic
embeddings.
\end{rem}

\subsection{Endo-submodules of $\Lambda$-modules}

\label{sec-endo}

Given a $\Lambda$-module $B$, 
we obtain a precise description for the submodules of 
$B$ when considered as a module over its endomorphism ring.

\medskip
We recall from   \cite[Theorem~24]{kap} that a height sequence $(m_i)$
defines an $\End(B)$ submodule 
consisting of all elements of $B$
with height sequence $\geq(m_i)$ (this is the ``fully transitive'' property).
The partial ordering on height sequences is given by
$(m_i)\geq (q_i)$ if $m_i\geq q_i$ holds for all $i\in\mathbb N$.

\medskip
It is shown in \cite[Theorem~25]{kap} that
every $\End(B)$-submodule of $B$ is cyclic
(as an $\End_\Lambda(B)$-module),
and hence is determined 
uniquely, up to equality, by the height sequence of a generator.

\medskip
Suppose $a\in B$ has height sequence $(m_i)$ with gaps exactly after
$m_{i_1}>m_{i_2}>\cdots> m_{i_s}$.
We can specify the corresponding $\End(B)$-submodule $X$ explicitly.
Namely, with $a=\sum_j p^{\ell_j}b_{\beta_j}$ as above, also the summands $p^{\ell_j}b_{\beta_j}$ are in $X$,
and so are their images under endomorphisms of $B$.
Put $k_j=\beta_j-\ell_j$ for $1\leq j\leq s$, then
$$\End(B)\cdot a = \sum_{j=1}^s(\rad^{\ell_j}B\cap \soc^{k_j}B).$$

\medskip
With this notation,
the endo-submodule $\End(B)\cdot a$ can also be written as an intersection of sums:
$$\End(B)\cdot a=\rad^{\ell_1}B\cap(\soc^{k_1}B+\rad^{\ell_2}B)\cap\cdots\cap
(\soc^{k_{s-1}}B+\rad^{\ell_s}B)\cap \soc^{k_s}B$$
To verify this, use induction and the modular law.

\medskip
The following result leads to
a quick proof for the formula for the number of $\End(B)$-submodules
of a finite length $\Lambda$-module $B$.

\begin{cor}\label{cor-endo-submodules}
Suppose the $\Lambda$-module $B$ is given by a partition $\beta$.
Any two of the following sets are in one-to-one correspondence.
\begin{enumerate}
\item The set of $\End(B)$-submodules of $B$.
\item The set of embeddings of the form $((a)\subset B)$, up to isomorphy.
\item The set of LR-tableaux with outer shape $\beta$ 
  in which each entry occurs exactly once.
\end{enumerate}
As a consequence, 
the $\Lambda$-module $B=N_\beta$ has exactly
$\prod_i(1+\beta_i-\beta_{i+1})$ many endo-submodules.
\end{cor}

\begin{proof}
In each of the sets, the elements are given by height sequences.
The correspondence follows from Corollary~\ref{cor-tableaux},
and Proposition~\ref{prop-tableau-of-cyclic}, 2., together with the isomorphy
between the lattices of $\End(B)$-submodules of $B$ and of
height sequences of elements in $B$ in \cite[Theorem~25]{kap}.
For the formula we count the third set. 
Note that tableau properties require that each entry be the rightmost
entry in its row since there are no multiple entries.  Hence in the $i$-th column,
there may be between $0$ and $\beta_i-\beta_{i+1}$ many entries.
The expression is the product of the number of choices for the number 
of entries in each column.
\end{proof}


\bigskip
    {\bf Dedication.}
    The authors wish to dedicate this paper to
    Professor Jos\'e Antonio de la Pe\~na
    on the occasion of his 65th birthday.
    Throughout their careers,
    Professor de la Pe\~na has shown an active interest in
    their work, supported them through invitations to important conferences,
    and inspired them through his dedication to representation theory and his
    interest in applications.
    His exemplary paper \cite{bps} has been a model for
    this manuscript since combinatorial invariants control families 
    of modules which turn out to play a meaningful role geometrically.


\bigskip
Address of the authors:

\parbox[t]{5.5cm}{\footnotesize\begin{center}
              Faculty of Mathematics\\
              and Computer Science\\
              Nicolaus Copernicus University\\
              ul.\ Chopina 12/18\\
              87-100 Toru\'n, Poland\end{center}}
\parbox[t]{5.5cm}{\footnotesize\begin{center}
              Department of\\
              Mathematical Sciences\\ 
              Florida Atlantic University\\
              777 Glades Road\\
              Boca Raton, Florida 33431\end{center}}

\smallskip \parbox[t]{5.5cm}{\centerline{\footnotesize\tt justus@mat.umk.pl}}
           \parbox[t]{5.5cm}{\centerline{\footnotesize\tt markus@math.fau.edu}}

\end{document}